\documentclass[11pt]{amsart}
\usepackage{amsmath,amsthm,amsfonts}
\numberwithin{equation}{section}

\usepackage{amssymb}
\usepackage{rotating}
\usepackage{float}
\usepackage{graphics}
\usepackage{pstricks}
\usepackage{rotating}
\newtheorem{teo}{Theorem}[section]
\newtheorem{pro}[teo]{Proposition}

\newtheorem{lem}[teo]{Lemma}
\newtheorem{rem}{Remark}
\theoremstyle{definition}
\newtheorem{ejem}{Example}
\newtheorem{defi}{Definition}[section]

\DeclareMathOperator{\Int}{Int}

\newcommand{\R}{\mathbb{R}} 
\newcommand{\N}{\mathbb{N}} 

\newcommand{\tlim}{\displaystyle\lim}

\newcommand{\defrac}{\displaystyle\frac}

\title[T-Zamfirescu and  T-weak contraction ]{T-Zamfirescu and  T-weak contraction mappings on cone metric spaces}

\author[J.R. Morales and E.M. Rojas]{Jos\'e R. Morales \and Edixon Rojas}

\address{Department of Mathematics, Faculty of Science, University of Los Andes, M\'erida-5101, Venezuela.}

\email{moralesj@ula.ve}

\email{edixonr@ula.ve}


\date{}

\begin{document}

\begin{abstract}

The purpose of this paper is to obtain sufficient conditions for the existence of a unique fixed point of T-Zamfirescu and T-weak
contraction mappings in the framework of complete cone metric spaces.
\end{abstract}

\maketitle

\section{Introduction}

In 2007, Guang and Xiang \cite{LoZh} generalized the concept of metric space, replacing the set of real numbers
by an ordered Banach space and defined a cone metric space. The authors there described the convergence of sequences in
cone metric spaces and introduced the completeness. Also, they proved some fixed point theorems of contractive mappings on
complete cone metric spaces. Since then, fixed point theorems for different (classic) classes of mappings on these spaces
  have been appeared, see for instance \cite{AbRh08}, \cite{IlRa09},  \cite{KaRaRa09}, \cite{KhSa08}, \cite{PaSh09},
 \cite{RaVa08} and \cite{ReHam08}.

On the other hand, recently  A. Beiranvand S. Moradi, M. Omid and H. Pazandeh \cite{BMOP} introduced the $T-$contraction
and $T-$contractive mappings and then they extended the Banach contraction principle and the Edelstein's fixed point Theorem.
S. Moradi \cite{MK} introduced the $T-$Kannan contractive mappings,  extending in this way the Kannan's fixed point theorem
\cite{Ka}. The corresponding version of $T$-contractive, $T$-Kannan mappings and $T-$Chalterjea contractions on cone metric
 spaces was studied in \cite{MR} and \cite{MR1} respectively.  In view of these facts, thereby the purpose of this paper is to study the
 existence of fixed points of $T-$Zamficescu and $T-$weak contraction mappings defined on a complete cone
  metric space $(M,d)$,  generalizing consequently the  results given in \cite{LoZh} and \cite{Za}.

\section{General framework}
In this section we recall the definition of cone metric space and some of their properties
 (see, \cite{LoZh}). The following notions will be useful for us in order to prove the main results.

\begin{defi}
Let $E$ be a real Banach space. A subset $P$ of $E$ is called a cone if and only if:
\begin{description}
\item[(P1)] $P$ is closed, nonempty and $P\neq \{0\}$;
\item[(P2)] $a,b\in \R,\,\, a,b\geq 0,\,\, x,y\in P$ imply $ax+by\in P$;
\item[(P3)] $x\in P$ and $-x\in P\Rightarrow x=0$. I.e., $P\cap(-P)=\{0\}$.
\end{description}
Given a cone $P\subset E,$ we define a partial ordering $\leq$ with respect to $P$ by $x\leq y$ if and only if
 $y-x\in P$. We write $x<y$ to indicate that $x\leq y$ but $x\neq y$,  while $x\ll y$ will stand for $y-x\in \Int P$. (interior of $P$.)
\end{defi}

\begin{defi}
Let $E$ be a Banach space and $P\subset E$ a cone. The cone $P$ is called normal if there is a number $K>0$ such that for all $x,y\in E,\,\, 0\leq x\leq y$ implies $\|x\|\leq K\|y\|.$ The least positive number satisfying the above is called the normal constant of $P.$
\end{defi}
In the following, we always suppose that $E$ is a Banach space, $P$ is a cone in $E$ with $\Int P\neq \emptyset$ and $\leq$ is partial ordering with respect to $P$.

\begin{defi}[\cite{LoZh}]
Let $M$ be a nonempty set. Suppose that  the mapping $d: M\times M\longrightarrow E$ satisfies:
\begin{description}
\item[(d1)] $0<d(x,y)$ for all $x,y\in M$, and $d(x,y)=0$ if and only if $x=y$;
\item[(d2)] $d(x,y)=d(y,x)$ for all $x,y\in M$;
\item[(d3)] $d(x,y)\leq d(x,z)+d(z,y)$ for all $x,y,z\in M$.
\end{description}
Then, $d$ is called a cone metric on $M$ and $(M,d)$ is called a cone metric space.
\end{defi}

Note that the notion of  cone metric space is more general that the concept of  metric space.

\begin{defi}
Let $(M,d)$ be a cone metric space. Let $(x_n)$ be a sequence in $M$ and $x\in M$.
\begin{itemize}
\item[(i)] $(x_n)$ converges to $x$ if for every $c\in E$ with $0\ll c$ there is an $n_0$ such that
           for all $n>n_0,\,\, d(x_n,x)\ll c.$ We denote this by $\tlim_{n\rightarrow \infty} x_n=x$ or $x_n\rightarrow x,\,\, (n\rightarrow \infty)$.
\item[(ii)] If for any $c\in E$ with $0\ll c$ there is an $n_0$ such that for all $n,m\geq n_0$, $\;d(x_n,x_m)\ll c$, then $(x_n)$ is called a Cauchy sequence in $M$.
\end{itemize}
Let $(M,d)$ be a cone metric space. If every Cauchy sequence is convergent in $M,$ then $M$ is called a complete cone metric space.
\end{defi}

\begin{lem}[\cite{LoZh}]
Let $(M,d)$ be a cone metric space, $P\subset E$ a normal cone with normal constant $K.$ Let $(x_n),\,\, (y_n)$ be sequences in $M$ and $x,y\in M$.
\begin{itemize}
\item[\textup{(i)}] $(x_n)$ converges to $x$ if and only if $\tlim_{n\rightarrow \infty} d(x_n,x)=0$.
\item[\textup{(ii)}] If $(x_n)$ converges to $x$ and $(x_n)$ converges to $y$, then $x=y$.
\item[\textup{(iii)}] If $(x_n)$ converges to $x$, then $(x_n)$ is a Cauchy sequence.
\item[\textup{(iv)}] $(x_n)$ is a Cauchy sequence if and only if $\tlim_{n,m\rightarrow \infty} d(x_n,x_m)=0$.
\item[\textup{(v)}] If $x_n\longrightarrow x$ and $y_n\longrightarrow y,\,\, (n\rightarrow \infty)$, then
           $d(x_n,y_n)\longrightarrow d(x,y)$.
\end{itemize}
\end{lem}

\begin{defi}
Let $(M,d)$ be a cone metric space, $P$ a normal cone with normal constant $K$ and $T: M\longrightarrow M$. Then
\begin{itemize}
\item[(i)] $T$ is said to be continuous, if $\tlim_{n\rightarrow \infty} x_n=x$ implies that
           $\tlim_{n\rightarrow \infty} T(x_n)=T(x)$ for all $(x_n)$ and $x$ in $M$.

\item[(ii)] $T$ is said to be subsequentially convergent if we have, for every sequence $(y_n),$ if
            $T(y_n)$ is convergent, then $(y_n)$ has a convergent subsequence.

\item[(iii)] $T$ is said to be sequentially convergent if we have, for every sequence $(y_n),$ if
             $T(y_n)$ is convergent then $(y_n)$ also is convergent.
\end{itemize}
\end{defi}

Examples of cone metric spaces can be found for instance in \cite{LoZh}, \cite{ReHam08} and references therein.

\section{Main Results}

This section is devoted to give fixed point results  for $T$-Zamfirescu and $T$-weak contraction mappings on complete (normal) cone metric spaces, as well as, their asymptotic behavior. First, we recall the following classes of contraction type mappings:
\begin{defi}\label{classes}
Let $(M,d)$ be a cone metric space and $T,S: M\longrightarrow M$ two mappings
\begin{itemize}
\item[(i)] The mapping $S$ is called a $T-$Banach contraction, (TB - Contraction) if there is $a\in
           [0,1)$ such that
\begin{equation*}\label{eq3.1}
d(TSx,TSy)\leq a d(Tx,Ty)
\end{equation*} for all $x,y\in M$.

\item[(ii)] The mapping $S$ is called a $T-$Kannan contraction, (TK - Contraction) if there is $b\in
            [0,1/2)$ such that
\begin{equation*}\label{eq3.2}
d(TSx,TSy)\leq b[d(Tx,TSx)+d(y,TSy)]
\end{equation*} for all $x,y\in M$.

\item[(iii)] A mapping $S$ is said to be a Chatterjea contraction, (TC - Contraction) if there is $c\in [0,1/2)$ such that
\begin{equation*}\label{eq3.3}
d(TSx,TSy)\leq c[d(Tx,TSy)+d(Ty,TSx)]
\end{equation*} for all $x,y\in M.$
\end{itemize}

It is clear that if we take $T=I_d$ (the identity map) in the Definition \ref{classes} we obtain the definitions of Banach contraction,
Kannan mapping (\cite{Ka}) and Chatterjea mapping (\cite{Cha}).
\end{defi}

Now, following the ideas of T. Zamfirescu \cite{Za} we introduce the notion of $T-$Zamfirescu mappings.

\begin{defi}\label{classes Zamfrescus}
Let $(M,d)$ be a cone metric space and $T,S: M\longrightarrow M$ two mappings.  $S$ is called a $T-$Zamfirescu mapping,
(TZ -mapping), if and only if, there are real numbers, $0\leq a<1,\,\, 0\leq b,c<1/2$ such that for all $x,y\in M,$ at least one of the next conditions are true:
\begin{description}
\item[($TZ_1$)] $d(TSx,TSy)\leq a d(Tx,Ty)$.

\item[($TZ_2$)] $d(TSx,TSy)\leq b[d(Tx,TSx)+d(Ty,TSy)]$.

\item[($TZ_3$)] $d(TSx,TSy)\leq c[d(Tx,TSy)+d(Ty,TSx)]$.
\end{description}
\end{defi}
If in Definition \ref{classes Zamfrescus} we take $T=I_d$ and $E=\R_{+}$ we obtain the definition of T. Zamfirescu \cite{Za}.

\begin{lem}\label{lem:equiv}
Let $(M,d)$ be a cone metric space and $T,S: M\longrightarrow M$ two mappings. If $S$ is a $TZ-$mapping, then there is $0\leq \delta<1$ such that
\begin{equation}\label{equ:lemma}
d(TSx,TSy)\leq \delta d(Tx,Ty)+2\delta d(Tx,TSx)
\end{equation}
 for all $x,y\in M$.
\end{lem}
\begin{proof}
If $S$ is a $TZ-$mapping, then at least one of $(TZ_1)$, $(TZ_2)$ o $(TZ_3)$ condition is true.

If $(TZ_2)$ holds, then:
\begin{equation*}
\begin{array}{ccl}
d(TSx,TSy) &\leq& b[d(Tx,TSx)+d(Ty,TSy)]\\ \\ &\leq& b[d(Tx,TSx)+d(Ty,Tx)+d(Tx,TSx)+d(TSx,TSy)]
\end{array}
\end{equation*}
 thus,
\begin{equation*}
(1-b)d(TSx,TSy)\leq b d(Tx,Ty)+2b d(Tx,TSx).
\end{equation*}
 From the fact that $0\leq b<1/2$ we get:
\begin{equation*}
d(TSx,TSy)\leq \defrac{b}{1-b}d(Tx,Ty)+\defrac{2b}{1-b}d(Tx,TSx).
\end{equation*}
with $\frac{b}{1-b}<1$. If $(TZ_3)$ holds, then similarly we get
\begin{equation*}
d(TSx,TSy)\leq \defrac{c}{1-c}d(Tx,Ty)+\defrac{2c}{1-c}d(Tx,TSx).
\end{equation*}
 Therefore,  denoting by
\begin{equation*}
\delta:=\max\left\{a,\, \defrac{b}{1-b},\, \defrac{c}{1-c}\right\}
\end{equation*}
 we have that $0\leq \delta<1$. Hence, for all $x,y\in M,$ the following inequality holds:
\begin{equation*}
d(TSx,TSy)\leq \delta d(Tx,Ty)+2\delta d(Tx,TSx).
\end{equation*}
\end{proof}

\begin{rem}
Notice that inequality \eqref{equ:lemma} in Lemma \ref{lem:equiv} can be replace by
\begin{equation*}
d(TSx,TSy)\leq \delta d(Tx,Ty)+2\delta d(Tx,TSy)
\end{equation*}
 for all $x,y\in M$.
\end{rem}

\begin{teo}\label{thm:characteristics}
Let $(M,d)$ be a complete cone metric space, $P$ be a normal cone with normal constant $K$. Moreover, let $T: M\longrightarrow M$ be a
continuous and one to one mapping and $S: M\longrightarrow M$ a $T-$Zamfirescu continuous mapping. Then
\begin{itemize}
\item[\textup{(i)}] For every $x_0\in M$,
\begin{equation*}
\tlim_{n\rightarrow \infty} d(TS^nx_0, TS^{n+1}x_0)=0.
\end{equation*}

\item[\textup{(ii)}] There is $y_0\in M$ such that
\begin{equation*}
\tlim_{n\rightarrow \infty} TS^nx_0=y_0.
\end{equation*}

\item[\textup{(iii)}] If $T$ is subsequentially convergent, then $(S^nx_0)$ has a convergent subsequence.

\item[\textup{(iv)}] There is a unique $z_0\in M$ such that $Sz_0=z_0$.

\item[\textup{(v)}] If $T$ is sequentially convergent, then for each $x_0\in M$ the iterate sequence $(S^nx_0)$
           converges to $z_0$.
\end{itemize}
\end{teo}
\begin{proof}
\begin{itemize}

\item[(i)] Since $S$ is a $T-$Zamfirescu mapping, then by Lemma \ref{lem:equiv}, there exists $0<\delta<1$ such
           that
\begin{equation*}
d(TSx,TSy)\leq \delta d(Tx,Ty)+2\delta d(Tx,TSx)
\end{equation*}
for all $x,y\in M$.

Suppose $x_0\in M$ is an arbitrary point and the Picard iteration associated to $S,$ $\;(x_n)$ is defined by
\begin{equation*}
x_{n+1}=Sx_n=S^nx_0,\qquad n=0,1,2,\ldots.
\end{equation*}
 Thus,
\begin{equation*}
d(TS^{n+1}x_0, TS^nx_0)\leq h d(TS^nx_0, TS^{n-1}x_0)
\end{equation*}
 where $h=\defrac{\delta}{1-2\delta}<1$. Therefore, for all $n$ we have
\begin{equation*}
d(TS^{n+1}x_0, TS^nx_0)\leq h^n d(TSx_0, Tx_0).
\end{equation*}
 From the above, and the fact the cone $P$ is a normal cone we obtain that
 \begin{equation*}
\|d(TS^{n+1}x_0, TS^nx_0)\|\leq K h^n \|d(TSx_0, Tx_0)\|,
\end{equation*}
taking limit $n\longrightarrow\infty$ in the above inequality we can conclude that
\begin{equation*}
\tlim_{n\rightarrow \infty} d(TS^{n+1}x_0, TS^nx_0)=0.
\end{equation*}

\item[(ii)] Now, for $m,n\in \N$ with $m>n$ we get
\begin{equation*}
\begin{array}{ccl}
d(TS^mx_0, TS^nx_0) &\leq& (h^n+\ldots+h^{m-1})d(TSx_0, Tx_0)\\ \\ &\leq& \defrac{h^n}{1-h}d(TSx_0, Tx_0).
\end{array}
\end{equation*}
Again; as above, since $P$ is a normal cone we obtain
\begin{equation*}
\tlim_{n,m\rightarrow \infty} d(TS^mx_0, TS^nx_0)=0.
\end{equation*}
 Hence, the fact that $(M,d)$ is a complete cone metric space, imply that $(TS^nx_0)$ is a Cauchy sequence in $M$, therefore there is $y_0\in M$ such that
\begin{equation*}
\tlim_{n\rightarrow \infty} TS^nx_0=y_0.
\end{equation*}

\item[(iii)] If $T$ is subsequentially convergent, $(S^nx_0)$ has a convergent subsequence, so there is
             $z_0\in M$ and $(n_k)_{k=1}^{\infty}$ such that
\begin{equation*}
\tlim_{k\rightarrow \infty} S^{n_k}x_0=z_0.
\end{equation*}

\item[(iv)] Since $T$ and $S$ are continuous mappings we obtain:
\begin{equation*}
\tlim_{k\rightarrow \infty} TS^{n_k}x_0=Tz_0,\qquad \tlim_{k\rightarrow \infty} TS^{n_k+1}x_0=TSz_0
\end{equation*}
therefore,  $Tz_0=y_0=TSz_0,$ and since $T$ is one to one, then $Sz_0=z_0.$ So $S$ has a fixed point.

Now, suppose that $Sz_0=z_0$ and $Sz_1=z_1$.
\begin{equation*}
\begin{array}{ccl}
d(TSz_0,TSz_1) &\leq& \delta d(Tz_0,Tz_1)+2\delta d(Tz_0, TSz_0)\\ \\ d(Tz_0, Tz_1) &\leq& \delta d(Tz_0, Tz_1)
\end{array}
\end{equation*}
from the fact that $0\leq\delta<1$ and that $T$ is one to one we obtain that $z_0=z_1$.

\item[(v)] It is clear that if $T$ is sequentially convergent, then for each $x_0\in M$, the iterate
           sequence $(S^nx_0)$ converges to $z_0$.
\end{itemize}
\end{proof}

In 2003, V. Berinde (see, \cite{Be}, \cite{Ber}) introduced a new class of contraction mappings on metric spaces,
 which are called weak contractions. We will extend these kind of mappings by introducing a new function $T$ and we define
  it in the framework of cone metric spaces.

\begin{defi}
Let $(M,d)$ be a cone metric space and $T,S: M\longrightarrow M$ two mappings. $S$ is called a $T-$weak contraction,
(TW- Contraction, $T_{(S,L)}-$Contraction), if there exist a constant $\delta\in (0,1)$ and some $L\geq 0$ such that
\begin{equation*}
d(TSx,TSy)\leq \delta d(Tx,Ty)+L d(Ty,TSx)
\end{equation*}
for all $x,y\in M$.
\end{defi}

It is clear that if we take $T=I_d$ and $E=\R_{+}$ then we obtain the notion of Berinde \cite{Be}.

Due to the symmetry of the metric, the $T-$weak contractive condition implicitly include the following dual one:
\begin{equation*}
d(TSx, TSy)\leq \delta d(Tx, Ty)+L d(Tx,TSy)
\end{equation*}
for all $x,y\in M$.

The next proposition gives examples of $T-$weak contraction and it proof is  similar to the proof of Lemma \ref{lem:equiv}.

\begin{pro}
Let $(M,d)$ be a cone metric space and $T,S: M\longrightarrow M$ two mappings.
\begin{itemize}
\item[\textup{(i)}] If $S$ is a TB - contraction, then $S$ is a $T-$weak contraction.

\item[\textup{(ii)}] If $S$ is a TK - contraction, then $S$ is a $T-$weak contraction.

\item[\textup{(iii)}] If $S$ is a TC - contraction, then $S$ is a $T-$weak contraction.

\item[\textup{(iv)}] If $S$ is TZ - mapping, then $S$ is a $T-$weak contraction.
\end{itemize}
\end{pro}
Now we have the following result:

\begin{teo}
Let $(M,d)$ be a complete cone metric space, $P$  a normal cone with normal constant $K$. Let furthermore $T: M\longrightarrow M$  a continuous and one to one  mapping and  $S: M\longrightarrow M$ a continuous $T-$weak contraction. Then
\begin{itemize}
\item[\textup{(i)}] For every $x_0\in M$,
\begin{equation*}
\tlim_{n\rightarrow \infty} d(TS^nx_0, TS^{n+1}x_0)=0.
\end{equation*}

\item[\textup{(ii)}] There is $y_0\in M$ such that
\begin{equation*}
\tlim_{n\rightarrow \infty} TS^nx_0=y_0.
\end{equation*}

\item[\textup{(iii)}] If $T$ is subsequentially convergent, then $(S^nx_0)$ has a convergent subsequence.

\item[\textup{(iv)}] There is $z_0\in M$ such that
\begin{equation*}
Sz_0=z_0.
\end{equation*}

\item[\textup{(v)}] If $T$ is sequentially convergent, then for each $x_0\in M$ the iterate sequence $(S^nx_0)$
           converges to $z_0$.
\end{itemize}
\end{teo}
\begin{proof}
Similar to the proof of Theorem \ref{thm:characteristics}.
\end{proof}

As we see in Theorem \ref{thm:characteristics}, a $T-$Zamfirescu mapping has a unique fixed point. The next example shows that a $T-$weak contraction may has infinitely fixed points.

\begin{ejem}[\cite{Beri}]

Let $M=[0,1]$ be the unit interval with the usual metric and $T,S: M\longrightarrow M$ the identity maps, that is, $Tx=Sx=x$ for all $x\in M$. Then, taking $0\leq a<1$ and $L\geq 1-a$ we obtain
\begin{equation*}
\begin{array}{ccl}
d(TSx, TSy) &=& |TSx-TSy|\\ \\ |x-y| &\leq& a|x-y|+L|y-x|
\end{array}
\end{equation*}
 which is valid for all $x,y\in [0,1]$. Thus the set of the fixed points $F_S$ of the map $S$ is the interval $[0,1]$. I.e.,
\begin{equation*}
F_{S}=\{x\in [0,1]\,/\, Sx=x\}=[0,1].
\end{equation*}
\end{ejem}

It is possible to force the uniqueness of the fixed point of a $T-$weak contraction by imposing an additional contractive condition, as is shown in the next theorem.

\begin{teo}
Let $(M,d)$ be a complete cone metric space, $P$ be a normal cone with normal constant $K$. Let furthermore $T: M\longrightarrow M$  a continuous and one to one mapping and $S: M\longrightarrow M$  a $T-$weak contraction for which there is $\theta\in (0,1)$ and some $L_1\geq 0$ such that
\begin{equation*}
d(TSx, TSy)\leq \theta d(Tx,Ty)+L_1 d(Tx, TSx)
\end{equation*}
 for all $x,y\in M$. Then:
\begin{itemize}
\item[\textup{(i)}] For every $x_0\in M$
\begin{equation*}
\tlim_{n\rightarrow \infty} d(TS^nx_0, TS^{n+1}x_0)=0.
\end{equation*}

\item[\textup{(ii)}] There is $y_0\in M$ such that
\begin{equation*}
\tlim_{n\rightarrow \infty} TS^nx_0=y_0.
\end{equation*}

\item[\textup{(iii)}] It $T$ is subsequentially convergent, then $(S^nx_0)$ has a convergent subsequence.

\item[\textup{(iv)}] There is a unique $z_0\in M$ such that
\begin{equation*}
Sz_0=z_0.
\end{equation*}

\item[\textup{(v)}] If $T$ is sequentially convergent, then for each $x_0\in M$ the iterate sequence $(S^nx_0)$
           converges to $z_0.$
\end{itemize}
\end{teo}
\begin{proof}
 Assume $S$ has two distinct fixed points $x^{*}, y^{*}\in M.$ Then
\begin{equation*}
d(Tx^{*},Ty^{*})=d(TSx^{*}, TSy^{*})\leq \theta d(Tx^{*}, Ty^{*})+L_1 d(Tx^{*}, TSx^{*})
\end{equation*}
thus, we get
\begin{equation*}
d(Tx^{*}, Ty^{*})\leq \theta d(Tx^{*}, Ty^{*})\Leftrightarrow (1-\theta)d(Tx^{*}, Ty^{*})\leq0.
\end{equation*}
Therefore, $d(Tx^{*}, Ty^{*})=0$. Since $T$ is one to one, then $x^{*}=y^{*}$.

The rest of the proof follows as  the the proof of Theorem \ref{thm:characteristics}.
\end{proof}

\end{document}